\documentclass{amsart}

\usepackage{amsmath, amsthm,amssymb, amsfonts}
\usepackage{rotating}
\usepackage{url}
\usepackage{microtype}

\newcommand{\bt}{\beta}
\newcommand{\del}{\delta}
\newcommand{\gam}{\gamma}
\newcommand{\eps}{\epsilon}
\newcommand{\acal}{\mathcal{A}}
\newcommand{\rar}{\rightarrow}
\renewcommand{\phi}{\varphi}
\newcommand{\nint}{\mathbb{N}_0}
\newcommand{\bsl}{\setminus}
\newcommand{\qbinom}[2]{\genfrac{[}{]}{0pt}{}{#1}{#2}}

\def\Z{{\mathbb Z}}
\def\N{{\mathbb N}}

\newtheorem{thm}{Theorem}

\newtheorem{prop}[thm]{Proposition}
\newtheorem{cor}[thm]{Corollary}
\newtheorem{lemma}[thm]{Lemma}
\newtheorem*{defn}{Definition}

\theoremstyle{remark}
\newtheorem*{remark}{Remark}

\title{The Proportion of Weierstrass Semigroups}

\author[N. Kaplan]{Nathan Kaplan}
\address{Department of Mathematics\\
Harvard University\\
Cambridge, MA 02138}
\email{NKaplan@math.harvard.edu}
\author[L. Ye]{Lynnelle Ye}
\address{Department of Mathematics\\
Stanford University\\
Stanford, CA 94305}
\email{Lynnelle@stanford.edu}

\date{\today}



\begin{document}

\begin{abstract}
We solve a problem of Komeda concerning the proportion of numerical semigroups which do not satisfy Buchweitz' necessary criterion for a semigroup to occur as the Weierstrass semigroup of a point on an algebraic curve.  A result of Eisenbud and Harris gives a sufficient condition for a semigroup to occur as a Weierstrass semigroup.  We show that the family of semigroups satisfying this condition has density zero in the set of all semigroups.  In the process, we prove several more general results about the structure of a typical numerical semigroup.
\end{abstract}

\maketitle 

\section{Introduction}

A numerical semigroup $S$ is an additive submonoid of $\N$ such that $\N\setminus S$ is finite.  The complement is referred to as the gap set and is denoted by $H(S)$.  Its size is called the genus of $S$ and is usually denoted by $g(S)$.  The largest of these gaps is called the Frobenius number, denoted $F(S)$, and the smallest nonzero nongap is called the multiplicity, denoted $m(S)$.  When it will not cause confusion we will omit the $S$ and write $g,F$ and $m$.  A very good source for background on numerical semigroups is \cite{RGS}.

Let $C$ be a smooth projective algebraic curve of genus $g$ over the complex numbers.  It is a theorem of Weierstrass that given any $p \in C$ there are exactly $g$ integers $\alpha_i(p)$ with $1 = \alpha_1(p) < \cdots < \alpha_g(p) \le 2g-1$ such that there does not exist a meromorphic function $f$ on $C$ which has a pole of order $\alpha_i(p)$ at $p$ and no other singularities \cite{DC}.  This characterization makes it clear that the set $\N \setminus \{\alpha_1(p),\ldots, \alpha_g(p)\}$ is a numerical semigroup of genus $g$.  We say that a semigroup $S$ is Weierstrass if there exist some curve $C$ and some point $p\in C$ such that $S$ is this semigroup.  In the late 19th century, Hurwitz suggested studying which numerical semigroups are Weierstrass \cite{Hurwitz}.  

A point $p$ such that $(\alpha_1(p),\ldots, \alpha_g(p)) \neq (1,\ldots, g)$ is called a Weierstrass point of $C$, and it is known that there are at most $g^3-g$ such points.  It is an active area of research to consider a multiset $\mathcal{S}$ of at most $g^3-g$ semigroups of genus $g$ and study the set of curves for which $\mathcal{S}$ is the collection of semigroups of the Weierstrass points of the curve.  This multiset gives us important information about the geometry of the curve.  We would like to better understand, for example, the dimension of the moduli space of curves with a fixed collection of semigroups attached to their Weierstrass points.  For more on the history of this problem see the article of del Centina \cite{DC}, or the book \cite{ACGH}.  

 In this paper we focus on two criteria that address this problem of Hurwitz.  The first is a simple combinatorial criterion of Buchweitz which is necessary for a semigroup to occur as the Weierstrass semigroup of some point on some curve $C$ \cite{Buch}.  This condition gave the first proof that not all semigroups are Weierstrass. In the final section of the paper we consider a criterion of Eisenbud and Harris \cite{EH}, which shows that certain semigroups are Weierstrass.  These two simple criteria cover much of what we know about this problem.  The main result of this paper is to show that in some sense, both of the sets covered by these criteria have density zero in the entire set of numerical semigroups.  The overall proportion of Weierstrass semigroups remains completely unknown.

Let $N(g)$ be the number of numerical semigroups of genus $g$.  Recent work of Zhai \cite{Zhai}, building on work of Zhao \cite{Zhao}, gives a better understanding of the growth of $N(g)$.  These papers build towards resolving a conjecture of Bras-Amor\'os \cite{BA}.

\begin{thm}[Zhai]\label{Zhai}
The function $N(g)$ satisfies
\[\lim_{g\to\infty} \frac{N(g)}{\varphi^g} = C,\]
where $C>0$ is a constant and $\varphi = \frac{1+\sqrt{5}}{2}$ is the golden ratio.
\end{thm}
\noindent This result will play an important role in some of our proofs.

We next recall the criterion of Buchweitz \cite{Buch}.
\begin{prop}[Buchweitz]
Let $S$ be a semigroup of genus $g$ and let $H(S)$ be the set of gaps of $S$.  Suppose that there exists some $n > 1$ such that
\[|n H(S)| > (2n-1)(g-1),\]
where $n H(S)$ is the $n$-fold sum of the set $H(S)$.  Then $S$ is not Weierstrass.
\end{prop}

Let $NB(g)$ be the number of semigroups $S$ of genus $g$ for which there is some $n$ such that $S$ does not satisfy the Buchweitz criterion with this $n$.  Let $NB_2(g)$ be the number of semigroups $S$ of genus $g$ such that $|2 H(S)| > 3(g-1)$.  Komeda seems to be the first to have studied $\lim_{g\rightarrow \infty} \frac{NB_2(g)}{N(g)}$ \cite{Komeda}.

The following is part of a table included in \cite{Komeda}:
\[
\begin{tabular}{| c | c | c | c |}
\hline
$g$ & $N(g)$ & $NB_2(g)$ & $\frac{NB_2(g)}{N(g)}$ \\
\hline
\hline
$16$ & $4806$ & $2$ & $.000416$ \\
\hline
$17$ & $8045$ & $6$ & $.000746$ \\
\hline
$18$ & $13467$ & $15$ & $.001114$ \\
\hline
$19$ & $22464$ & $31$ & $.001380$ \\
\hline
$20$ & $37396$ & $67$ & $.001792$ \\
\hline
$21$ & $62194$ & $145$ & $.002331$ \\
\hline
$22$ & $103246$ & $293$ & $.002838$ \\
\hline
$23$ & $170963$ & $542$ & $.003170$ \\
\hline
$24$ & $282828$ & $1053$ & $.003723$ \\
\hline
$25$ & $467224$ & $1944$ & $.004161$ \\
\hline
\end{tabular}\ .
\]
One main goal of this paper is to show that this limit is $0$.  In fact, we will show that the limit of the ratio of $NB(g)$ to $N(g)$ is $0$.  The key step in this argument will be a technical result building on work of Zhai \cite{Zhai}.

In the final part of the paper we will focus on the proportion of semigroups which are known to occur as Weierstrass semigroups.  Eisenbud and Harris \cite{EH}, have shown that a certain class of semigroups with $F<2m$ do occur as Weierstrass semigroups.  We will show that the proportion of such semigroups is $0$ as $g$ goes to infinity.

\section{Semigroups Satisfying the Buchweitz Criterion}

We first show that certain classes of semigroups cannot possibly fail the Buchweitz criterion for any $n$.  Fix $\varepsilon > 0$ and suppose that $S$ is a semigroup with $(2-\varepsilon) m < F < (2+\varepsilon) m$.  We want to consider when $S$ fails the Buchweitz criterion for some chosen value of $n$.  We have,
\[ |n H| \le (2+\varepsilon)n m - (n-1).\]
Therefore, $|n H| > (2n-1)(g-1)$ implies that 
\[g \le \frac{(2+\varepsilon)n m + n}{2n-1} = (2+\varepsilon)\frac{nm}{2n-1}  + \frac{n}{2n-1}<\left(2+\frac{1}{20}\right) \frac{nm}{2n-1}\]
whenever $\varepsilon< \frac{1}{20}-\frac{1}{m}$.  We see that this inequality holds for any $\varepsilon< \frac{1}{21}$ and $m \ge 420$.

We note that since $n\ge 2$ is an integer, \[\left(2+\frac{1}{20}\right) \frac{n}{2n-1} \le \frac{41}{30} < 1.3667.\]

We will state the results of the above paragraph as a proposition.
\begin{prop}\label{Bfail}
Let $\varepsilon < \frac{1}{21}$ and $m\ge 420$.

Suppose that $S$ is a semigroup with $(2-\varepsilon) m < F < (2+\varepsilon) m$.  Then $|nH| > (2n-1)(g-1)$ implies $g < 1.3667 m$.
\end{prop}

The main technical result of the rest of this paper is that the restriction on the genus in the proposition does not occur often.

\begin{thm}\label{tech}
\begin{enumerate}
\item Fix $\varepsilon > 0$.  Let $A(g)$ be the number of semigroups of genus $g$ satisfying $(2-\varepsilon) m < F < (2+\varepsilon) m$.  Then
\[\lim_{g \to \infty} \frac{A(g)}{N(g)} = 1.\]

\item Let $B(g)$ be the number of semigroups of genus $g$ with $m < 420$.  Then
\[\lim_{g\to \infty} \frac{B(g)}{N(g)} = 0.\]

\item Let $C(g)$ be the number of semigroups of genus $g$ with $g < 1.3667 m$.  Then
\[\lim_{g \to \infty} \frac{C(g)}{N(g)} = 0.\]
\end{enumerate}
\end{thm}

The claim for $A(g)$ follows directly from Proposition~\ref{minuseps} and Theorem~\ref{pluseps} below.  Suppose we have established this claim.  The proportion of numerical semigroups of genus $g$ with Frobenius number at least $(2+\varepsilon) m$ goes to $0$ as $g$ goes to infinity.  For any $\varepsilon$ the number of semigroups with Frobenius number at most $(2+\varepsilon) 420$ is finite.  Therefore, as $g$ goes to infinity, almost all semigroups satisfying $(2-\varepsilon) m < F < (2+\varepsilon) m$ have $m > 420$.  This establishes the claim for $B(g)$.  The statement regarding $C(g)$ follows from Corollary~\ref{msizeless2m} and Proposition~\ref{fracmg}.

From Theorem \ref{tech} it is easy to prove our main theorem.
\begin{thm}
\label{main}
Let $NB(g)$ be the number of semigroups of genus $g$ which fail the Buchweitz criterion for some $n$.
Then
\[\lim_{g \to \infty} \frac{NB(g)}{N(g)} = 0.\]
\end{thm}

\begin{proof}
Suppose Theorem \ref{tech} holds.  Choose $\varepsilon < \frac{1}{21}$.  Theorem \ref{tech} implies that almost all semigroups have Frobenius number and multiplicity in the range given in the statement of Proposition \ref{Bfail}, but that almost no such semigroups satisfy $g < 1.3667 m$, completing the proof.
\end{proof}

\section{Ap\'ery Sets and Semigroups with $F<2m$}
The Ap\'ery set of a numerical semigroup $S$ with respect to its multiplicity $m$, often just called the Ap\'ery set, is a set of $m$ nonnegative integers giving for each $0\le i \le m-1$ the smallest integer in $S$ congruent to $i$ modulo $m$ \cite{RGS}.  We will omit $0$ from this set, and represent the Ap\'ery set by $\{k_1 m +1,\ldots, k_{m-1} m + m-1\}$ where each $k_i \in \N$.  We note that there are exactly $k_i$ gaps of $S$ equivalent to $i$ modulo $m$, and therefore the genus of $S$ is $\sum_{i=1}^{m-1} k_i$.  The Frobenius number is the largest Ap\'ery set element minus $m$.

From the definition of the Ap\'ery set it is clear that certain inequalities must hold between the integers $k_i$.  In fact, a result of Branco, Garc\'ia-Garc\'ia, Garc\'ia-S\'anchez and Rosales (\cite{Rosales}) gives a set of inequalities which completely determine whether the set $\{k_1 m+1, \ldots, k_{m-1} m + m-1\}$ is the Ap\'ery set of a numerical semigroup of multiplicity $m$.

\begin{prop}[Rosales et al.]\label{Rosales}
Consider the following set of inequalities:
\begin{eqnarray*}
x_i \ge 1 & & \text{for all}\ i\in \{1,\ldots, m-1\} \\
x_i + x_j \ge x_{i+j}  & &\text{for all}\ 1 \le i \le j \le m-1,\ i+j \le m-1\\
x_i +x_j + 1 \ge x_{i+j-m} & &\text{for all}\ 1 \le i \le j \le m-1,\ i+j > m\\
x_i \in \Z & &\text{for all}\ i\in \{1,\ldots, m-1\},\\
\sum_{i=1}^{m-1} x_i = g.
\end{eqnarray*}

There is a one-to-one correspondence between semigroups with multiplicity $m$ and genus $g$ and solutions to the above inequalities, where we identify the solution $\{k_1,\ldots, k_{m-1}\}$ with the semigroup that has Ap\'ery set $\{k_1 m + 1,\ldots, k_{m-1} m + m-1\}$.
\end{prop}

Recent work has made use of this correspondence, counting semigroups by counting valid Ap\'ery sets, for example \cite{Kaplan}.  This can be very useful in giving numerical results.  We will separately consider two classes of semigroups, those with $F< 2m$ and those with $2m < F< 3m$.  The first case is much simpler.

Note that $F<2m$ is exactly equivalent to the condition that each $k_i$ is equal to $1$ or $2$.  Given $m$, the above proposition implies that any set $\{k_1,\ldots, k_{m-1}\}$ where each $k_i$ is either $1$ or $2$ gives the Ap\'ery set of some numerical semigroup.

Suppose we have a semigroup of genus $g$ with $F<2m$ and Ap\'ery set given by $\{k_1m + 1,\ldots, k_{m-1} m + m-1\}$.  Let $R$ be the number of $k_i$ which are equal to $2$.  We see that $m - 1 + R = g$ and that $R$ can take on any value from $0$ to $m-1$.  Therefore we have that $m$ can take on any value from $\lfloor \frac{g+2}{2} \rfloor$ to $g+1$.  Given a pair of $m$ and $R$ such that $m-1 +R = g$ we see that there are $\binom{m-1}{R} = \binom{g-R}{R}$ choices for the $R$ values of $i$ such that $k_i = 2$.  It is a straightforward inductive exercise to prove that 
\[\sum_{R = 0}^{\lfloor \frac{g}{2}\rfloor} \binom{g-R}{R} = F_{g+1},\]
the $g+1$st Fibonacci number.

It is well-known that $F_{g+1}$ is asymptotic to $\frac{\varphi}{\sqrt{5}} \varphi^g = \frac{5+\sqrt{5}}{10} \varphi^g$ as $g$ goes to infinity.  The above sum is very tightly clustered around its maximum value.
\begin{prop}
\label{less2m}
Let $\alpha = \frac{5-\sqrt{5}}{10}$ and fix $\varepsilon > 0$.  We have
\[ \sum_{R = 0}^{\lceil(\alpha-\varepsilon) g\rceil} \binom{g-R}{R} = o(\varphi^g).\]
We also have
\[ \sum_{R = \lfloor(\alpha + \varepsilon) g\rfloor}^{\lfloor \frac{g}{2}\rfloor} \binom{g-R}{R} = o(\varphi^g).\]
\end{prop}

\begin{proof}
Stirling's approximation says that $n! \sim \sqrt{2\pi n} \left(\frac{n}{e}\right)^n$.  Therefore, 
\begin{eqnarray*}
\binom{(1-c)n}{cn} & = & \frac{((1-c)n)!}{(cn)! ((1-2c)n)!} \\
& \sim & \frac{\sqrt{2\pi (1-c) n}}{\sqrt{2\pi (1-2c) n}\sqrt{2\pi c n}} \frac{e^{cn} e^{(1-2c)n}}{e^{(1-c)n}} \frac{((1-c)n)^{(1-c)n}}{((1-2c)n)^{(1-2c)n} (cn)^{cn}} \\
& = & \frac{1}{\sqrt{2 \pi n}} \frac{ \sqrt{1-c}}{\sqrt{c}\sqrt{1-2c}} \left(\frac{(1-c)^{1-c}}{c^c (1-2c)^{1-2c}}\right)^n.
\end{eqnarray*}
We see that $\binom{(1-c)n}{cn}$ is asymptotic to a constant depending on $c$ divided by $\sqrt{n}$ times $\left(\frac{(1-c)^{1-c}}{c^c (1-2c)^{1-2c}}\right)^n$.

Let $f(c) = \frac{(1-c)^{1-c}}{c^c (1-2c)^{1-2c}}$.  We claim that $f$ attains its maximum value in the range from $0$ to $\frac{1}{2}$ at $\frac{5-\sqrt{5}}{10}$.  We instead find the maximum value of $\ln(f(c))$ in this range.  We can see that the derivative of $\ln(f(c))$ is $2 \ln(1-2c) - \ln(1-c) -\ln(c)$.  Taking an exponential, we see that this is equal to $0$ when $\frac{(1-2c)^2}{(1-c)c} = 1$.  This gives $5c^2 - 5c + 1= 0$, which has roots at $c = \frac{5 \pm \sqrt{5}}{10}$.  Only one of these roots occurs in the range for which $c<1-c$, meaning that this is our unique critical point in this interval.  By choosing any value of $c$ between $0$ and this critical point, for example $c = \frac{1}{4}$, we see that $f(c)$ is increasing in the range from $0$ to our critical value, showing that $f(c)$ attains a maximum at $c = \frac{5-\sqrt{5}}{10}$.  At this value, $f(c) = \varphi$, the golden ratio.

Therefore,  
\[ \sum_{R = 0}^{\lceil(\alpha-\varepsilon) g\rceil} \binom{g-R}{R} \le (g+1) \binom{ g- \lfloor(\alpha- \varepsilon/2)g\rfloor}{\lfloor (\alpha-\varepsilon/2)g \rfloor},\]
for $g$ sufficiently large.  This is asymptotic to a constant depending on $\varepsilon$ times $\sqrt{g}$ times $\left(\frac{(1-c)^{1-c}}{c^c (1-2c)^{1-2c}}\right)^g$ for $c = \alpha-\varepsilon/2$.  This last term is $r^g$ where $r< \varphi$, showing that this sum is $o(\varphi^g)$.

We also have
\[ \sum_{R = \lfloor(\alpha + \varepsilon) g\rfloor}^{\lfloor \frac{g}{2}\rfloor} \binom{g-R}{R}  \le (g+1) \binom{ g-\lceil(\alpha+ \varepsilon/2)g\rceil}{\lceil(\alpha+\varepsilon/2)g\rceil},\]
for sufficiently large $g$.  This is also asymptotic to a constant depending on $\varepsilon$ times $\sqrt{g}$ times some $r^g$ where $r < \varphi$.  This sum is therefore also $o(\varphi^g)$.
\end{proof}

We first state a corollary related to the ratio of the multiplicity to the genus of a semigroup with $F<2m$.

\begin{cor}
\label{msizeless2m}
Fix $\epsilon >0$ and $\gamma = \frac{5+\sqrt{5}}{10}$.  Let $E_\epsilon(g)$ be the number of numerical semigroups with $F<2m$ and $(\gamma-\epsilon) g < m < (\gamma+\epsilon) g$.  Let $I(g)$ be the number of numerical semigroups with $F<2m$.  Then $\lim_{g\to \infty} \frac{E_\epsilon(g)}{I(g)} = 1$.
\end{cor}

\begin{proof}
We have $g = m-1+R$ and have seen that almost all semigroups with $F<2m$ have $(\alpha-\epsilon) g < R < (\alpha +\epsilon) g$ for $\alpha = \frac{5-\sqrt{5}}{10}$.  Since $1 - \alpha = \gamma$, we see that almost all semigroups with $F<2m$ have $(\gamma-\epsilon) g +1 < m < (\gamma+\epsilon) g + 1$.  Taking $g$ to infinity completes the proof.
\end{proof}

This proposition implies that for $g$ sufficiently large, almost all semigroups $S$ with genus $g$ and $F<2m$ have genus $m-1+R$ close to $m-1+ \alpha g$.  We note that when $R > \frac{g}{4} + 1 $, since $m-1+ R = g$ we have $m-1 < \frac{3g}{4} -1$.  Also note that $\frac{5-\sqrt{5}}{10} > \frac{1}{4}$.  In this case, we see that since our largest gap is at most $2m-1$, we have 
\[|nH| < n(2m-1)-(n-1) = 2n(m-1) + 1 < \frac{3n}{2} g - (2n- 1) \le (2n-1)(g-1),\]
since $n \ge 2$.  Therefore, we see that almost no semigroup with $F<2m$ fails the Buchweitz criterion for any $n$.  We have proven the following, the easy part of our main result.

\begin{prop}
Let $D(g)$ be the number of semigroups of genus $g$ with $F<2m$ and which fail the Buchweitz criterion for some $n$.  Then
\[\lim_{g\to \infty} \frac{D(g)}{N(g)} = 0.\]
\end{prop}

The following proposition will play a part in the proof of Theorem \ref{tech}.

\begin{prop}
\label{minuseps}
Let $\epsilon > 0$.  Let $N^*_{\eps}(g)$ be the number of semigroups of genus $g$ with $F \le (2-\epsilon) m$.  We have
\[\lim_{g\rightarrow\infty} \frac{N^*_{\eps}(g)}{N(g)}= 0.\]

\end{prop}

\begin{proof}
We note that $N^*_{\epsilon}(g) = \sum_{R=0}^{\lfloor \frac{g}{2}\rfloor} \binom{\lfloor(1-\epsilon)(g-R)\rfloor}{R}$.  We note that for $\epsilon'<\epsilon/2$ we have $(1-\epsilon)(g-R)<(1-\epsilon')g - R$ since $\frac{\epsilon}{\epsilon-\epsilon'}< 2 \le \frac{g}{R}$.  Therefore, for sufficiently large $g$ we get an upper bound for this sum that is $g+1$ times the maximum value of $\binom{\lfloor (1-\epsilon')g \rfloor - R}{R}$.  By the proof of Proposition \ref{less2m}, this value is asymptotic to some constant depending on $\epsilon'$ and $c$ divided by the square root of $(1-\epsilon')g$, times $r^g$  for some $r<\varphi$.  This shows that the sum is $o(\varphi^g)$.

\end{proof}

Finally, the following proposition will be convenient for the proof of Theorem~\ref{quad}.

\begin{prop}
\label{close1}
For any $\eps>0$, there exists a $\del>0$ so that 
\[
\sum_{R=0}^{\lceil \del g\rceil}\binom{g-R}{R}=o((1+\eps)^g).
\]
\end{prop}

\begin{proof}
Consider again the function $f(c) = \frac{(1-c)^{1-c}}{c^c (1-2c)^{1-2c}}$. Recall from the proof of Proposition~\ref{less2m} that $\binom{(1-c)g}{cg}$ is asymptotic to $f(c)^g$ times a constant depending on $c$ divided by $\sqrt g$. As $c$ approaches $0$, L'H\^opital's rule shows that $c\ln c=\frac{\ln c}{1/c}$ approaches $0$, and thus $c^c$ approaches $1$. Then $f(c)$ approaches $1$ as $c$ approaches $0$. The proposition follows directly.
\end{proof}

\section{Semigroups with $F > 2m$}

We first recall a recent result of Zhai \cite{Zhai}, building on work of Zhao \cite{Zhao}, that shows that we can focus on semigroups with $2m < F < 3m$.

\begin{thm}[Zhai]
\label{3m}
Let $L(g)$ be the number of semigroups with $F>3m$ and genus $g$.  Then
\[\lim_{g\to \infty} \frac{L(g)}{N(g)} = 0.\]
\end{thm}

In the rest of this section we will focus on the semigroups with $2m<F< 3m$ in more detail and show that for any $\varepsilon > 0$, the proportion of them with $(2+\varepsilon)m < F< 3m$ goes to zero as $g$ goes to infinity.

\begin{thm}
\label{pluseps}
Let $\varepsilon > 0$ and let $P_{\eps}(g)$ be the number of semigroups with $(2+\varepsilon) m < F< 3m$.  Then 
\[\lim_{g \to \infty} \frac{P_{\eps}(g)}{N(g)} = 0.\]
\end{thm}

We require the following concepts from Zhao~\cite{Zhao}. Let $\acal_k=\{A\subset[0,k-1]\mid0\in A\text{ and }k\notin A+A\}$. Let $S$ be a numerical semigroup with multiplicity $m$ and Frobenius number $F$ satisfying $2m<F<3m$. We say that $S$ has type $(A;k)$, where $0<k<m$ and $A\in\acal_k$, if $F=2m+k$ and $S\cap[m,m+k]=A+m$. Every numerical semigroup with $2m<F<3m$ has a unique type $(A;k)$, since $k=F-2m$ and $A=S\cap[m,m+k]-m$. Zhao proves the following.

\begin{prop}[Zhao]
Let $k$ be a positive integer and let $A\in\acal_k$. Then the number of numerical semigroups of genus $g$ and type $(A;k)$ is at most 
\[
F_{g-|(A+A)\cap[0,k]|+|A|-k-1},
\]
where $F_a$ is the $a$th Fibonacci number.
\end{prop}

We now note that if a semigroup $S$ of type $(A;k)$ satisfies $(2+\eps)m<F<3m$, it must have $k>\eps m>\eps g/3$, since $g\le3(m-1)$. We also have the general fact that $F_a\le\frac{2}{\sqrt5}\phi^a$ for all $a$. Therefore
\begin{align*}
P_{\eps}(g) &\le\sum_{\eps g/3<k<g}\sum_{A\in\acal_k}F_{g-|(A+A)\cap[0,k]|+|A|-k-1}\\
&\le\frac{2}{\sqrt5}\sum_{\eps g/3<k<g}\sum_{A\in\acal_k}\phi^{g-|(A+A)\cap[0,k]|+|A|-k-1}
\end{align*}
implying that
\begin{align*}
P_{\eps}(g)\phi^{-g}&\le\frac{2}{\sqrt5}\sum_{\eps g/3<k<g}\sum_{A\in\acal_k}\phi^{-|(A+A)\cap[0,k]|+|A|-k-1}\\
&\le\frac{2}{\sqrt5}\sum_{k= \lceil \eps g/3 \rceil }^{\infty}\sum_{A\in\acal_k}\phi^{-|(A+A)\cap[0,k]|+|A|-k-1}.
\end{align*}
Let $T(g)$ be the number of numerical semigroups of genus $g$ satisfying $F<3m$. We have the following theorem from Zhao~\cite{Zhao}.

\begin{thm}[Zhao]
\[
\lim_{g\rar\infty}T(g)\phi^{-g}=\frac{\phi}{\sqrt5}+\frac{1}{\sqrt5}\sum_{k=1}^{\infty}\sum_{A\in\acal_k}\phi^{-|(A+A)\cap[0,k]|+|A|-k-1}.
\]
\end{thm}
We are now in a position to prove Theorem~\ref{pluseps}.

\begin{proof}[Proof of Theorem~\ref{pluseps}.]
By Theorem \ref{Zhai}, we know that $T(g)\phi^{-g}$ is bounded above, so the sum 
\[
\sum_{k=1}^{\infty}\sum_{A\in\acal_k}\phi^{-|(A+A)\cap[0,k]|+|A|-k-1}
\]
converges. It follows that $\sum_{k= \lceil \eps g/3\rceil}^{\infty}\sum_{A\in\acal_k}\phi^{-|(A+A)\cap[0,k]|+|A|-k-1}$ approaches $0$ as $g$ goes to infinity.
\end{proof}

\begin{prop}
\label{fracmg}
Let $\eps>0$ and $\gamma=\frac{5+\sqrt5}{10}$. Let $\Phi_{\eps}(g)$ be the number of numerical semigroups with genus $g$ and $\left(\gam-\eps\right)g<m<\left(\gam+\eps\right)g$. Then $\lim_{g\rar\infty}\frac{\Phi_{\eps}(g)}{N(g)}=1$.
\end{prop}

\begin{proof}
By Theorem~\ref{3m}, it suffices to consider the cases $m<F<2m$ and $2m<F<3m$. The first case is simply Corollary~\ref{msizeless2m}, so we now assume $2m<F<3m$.

The Ap\'ery set of a numerical semigroup with $2m< F< 3m$ is of the form $\{k_1 m + 1,\ldots, k_{m-1} m + (m-1)\}$ where each $k_i \in \{1,2,3\}$ and at least one is equal to $3$.  Let $a$ be maximal such that $k_a = 3$.  By Proposition \ref{Rosales}, the number of semigroups with multiplicity $m$ and $F = 2m+a$ is exactly equal to the number of sequences $(k_1,\ldots, k_{m-1})$ satisfying the following conditions:
\begin{enumerate}
 \item For each $1\le i \le a-1,\ k_i \in \{1,2,3\}$, 
\item $k_a = 3$,
\item For each $a+1 \le j \le m-1,\ k_j \in \{1,2\}$,
\item For each $i,j$ with $i+j \le m-1$ and $k_i = k_j = 1$ we have $k_{i+j} \neq 3$.
\end{enumerate}

Let $H(a,b)$ be the number of numerical semigroups with $2m<F<3m$, multiplicity $a+1$, and genus $b$. Then the number of possibilities for the sequence $(k_1,\ldots, k_a)$ satisfying $\sum_{i=1}^a k_i = b$ is simply $H(a,b)$. Since $\sum_{i=1}^{m-1} k_i = g$ and $F=2m+a$, the remaining elements $(k_{a+1},\ldots, k_{m-1})$ consist of $g-b-(m-1-a)$ values of $k_i$ equal to $2$, with the rest equal to $1$.  These can be arranged in any order. Thus the total number of numerical semigroups with $2m<F<3m$ is 
\[
\sum_{b=3}^g\sum_{a=\lceil b/3\rceil}^{b}\sum_{m=a+1}^{g}H(a,b)\binom{m-1-a}{g-b-(m-1-a)}.
\]
Applying Theorem~\ref{pluseps} with $\varepsilon = \epsilon/6$, we need only consider the case $a<\eps m/6$, for which $b\le3a<\eps m/2\le\eps g/2$. The number of such numerical semigroups is at most
\[
\sum_{b<\eps g/2}\sum_{a=\lceil b/3\rceil}^b\sum_{m=a+1}^{g}H(a,b)\binom{m-1-a}{g-b-(m-1-a)}.
\]
We need to show that those terms in the above sum for which $m$ is outside the range $((\gam-\eps)g,(\gam+\eps)g)$ contribute $o(\phi^g)$ to the sum. For such $m$, we have $|m-\gam g| \ge \eps g$, which implies that
\[
|m-1-a-\gam(g-b)+1+a-\gam b| \ge \eps g.
\]
By the triangle inequality, we have 
\begin{align*}
|m-1-a-\gam(g-b)|+1+a+\gam b&=|m-1-a-\gam(g-b)|+|1+a|+|-\gam b|\\
&\ge|m-1-a-\gam(g-b)+1+a-\gam b|.
\end{align*} 
Therefore, when $|m-\gam g | \ge \eps g$ we have
\begin{align*}
|m-1-a-\gam(g-b)|+1+a+\gam b &\ge \eps g\\
|m-1-a-\gam(g-b)| &\ge \eps g-1-a-\gam b\\
&\ge\eps g-1-\eps g/6-\gam\eps g/2\\
&\ge0.471\eps g\ge0.471\eps(g-b)
\end{align*}
for sufficiently large $g$. As in the proof of Proposition \ref{less2m}, for such $m$, there is some $\psi<\phi$ for which $\binom{m-1-a}{g-b-(m-1-a)}=O(\psi^{g-b})$. Since the total number of numerical semigroups of genus $b$ is asymptotic to $\phi^b$, we certainly have $H(a,b)=O(\phi^b)$. We conclude that for $a,b,m$ satisfying the above conditions, 
\[H(a,b)\binom{m-1-a}{g-b-(m-1-a)}\le c\psi^{g-b}\phi^b\le c_{\eps}(\psi^{1-\eps/2}\phi^{\eps/2})^g,\] so the total contribution to the sum from such $m$ is indeed $o(\phi^g)$.
\end{proof}

Proposition~\ref{fracmg} implies, in particular, that for fixed $\eps>0$, as $g$ approaches infinity, the proportion of numerical semigroups with $m>(\gam+\eps)g$ approaches $0$. For $\eps=\frac{1}{1.3667}-\gam\approx0.0081>0$, the property $m>(\gam+\eps)g$ is precisely $g<1.3667m$. This implies statement (3) of Theorem~\ref{tech}, and thus completes the proof of Theorem~\ref{main}.

We also point out the following corollary which answers a question from \cite{Kaplan}.

\begin{cor}
Let $R(g)$ be the number of semigroups of genus $g$ for which $2g<3m$.  Then 
\[\lim_{g\rightarrow\infty} \frac{R(g)}{N(g)} = 1.\]
\end{cor}
\begin{proof}
We note that $2g<3m$ is equivalent to $\frac{2g}{3} < m$.  Proposition \ref{fracmg} implies that for any $\epsilon>0$, almost all semigroups satisfy $(\gamma-\epsilon) g < m$.  Since $\gamma > .72$, this completes the proof.
\end{proof}

It is interesting that this limit is $1$ since the numerical evidence for small $g$ is inconclusive.  For example, for $g=24$ the value of this ratio is approximately $.3962$ and there is no clear trend toward $1$.  See the chart at the beginning of Section 6 of \cite{Kaplan}.  

This corollary has implications for Theorem 1 of \cite{Kaplan}.  Let $N(m,g)$ denote the number of semigroups $S$ with multiplicity $m$ and genus $g$.
\begin{thm}
Suppose $2g<3m$.  Then $N(m-1,g-1) + N(m-1,g-2) = N(m,g)$.
\end{thm}
This theorem combined with the previous corollary gives one way to understand why $N(g)$ grows like a constant times $\varphi^g$.  Unfortunately, this does not give a new proof of Zhai's result, Theorem \ref{Zhai}, because it is used in the proof of the corollary.

\section{Semigroups which do occur as Weierstrass semigroups}

We first recall the definition of the weight of a numerical semigroup.  This is another way of measuring a semigroup's complexity.
\begin{defn}
Let $S$ be a semigroup of genus $g$ with gap set $\{a_1,\ldots, a_g\}$.  We define the weight of $S$ by $\ W(S) = \sum_{i=1}^g a_i - \frac{g(g+1)}{2}$.
\end{defn}
Note that for any $g\ge 1$, the semigroup containing all positive integers greater than $g$ has weight zero.  This definition plays an important role in the theory of semigroups and algebraic curves since for any curve $C$ of genus $g$ it is known that the sum of the weights of all of the semigroups of Weierstrass points of $C$ is $g^3-g$ \cite{ACGH}.

The following difficult result of Eisenbud and Harris proves that certain semigroups do occur as the Weierstrass semigroup of a point on some curve \cite{EH}.

\begin{thm}
Let $S$ be a semigroup with $F<2m$ and $W(S) < g-1$.  Then $S$ is Weierstrass.
\end{thm}

We will show that this condition on the weight of $S$ is quite restrictive.
\begin{prop}
\label{lessg}
Let $Q(g)$ be the number of semigroups $S$ with $F<2m$ and $W(S) < g-1$.  Then
\[\lim_{g\to \infty} \frac{Q(g)}{N(g)} = 0.\]
\end{prop}

This proposition is actually a consequence of a much stronger statement about the weights of numerical semigroups. To prove that statement, we first need the following lemma.

\begin{lemma}
\label{qbinom}
Let $p(x,y,z)$ be the number of partitions of $x$ into at most $y$ parts, each of size at most $z$. Then the number of numerical semigroups with genus $g$, multiplicity $m$, and weight $w$ satisfying $m<F<2m$ is exactly $p(w-(g-m+1),g-m+1,2m-2-g)$.
\end{lemma}

\begin{remark}
Equivalently, this is the coefficient of $q^{w-(g-m+1)}$ in the $q$-binomial coefficient $\qbinom{m-1}{g-m+1}_q$. See~\cite{Stanley} for details. 
\end{remark}

\begin{proof}
Suppose $\nint\bsl S=\{1,2,\dotsc,m-1,m+i_1,\dotsc,m+i_{g-m+1}\}$ with $i_a\in[1,m-1]$ for all $a$. We have

\begin{align*}
w &=1+2+\dotsb+m-1+(m+i_1)+\dotsb+(m+i_{g-m+1})\\
&\hspace{0.2in}-(1+2+\dotsb+m-1+m+\dotsb+g) = \sum_{a=1}^{g-m+1}(i_a-a+1),
\end{align*}
which can be rearranged as
\begin{align*}
w-(g-m+1)&=\sum_{a=1}^{g-m+1}(i_a-a).
\end{align*}
The $i_a-a$ are nonnegative because $i_1 \ge 1$ and $i_{a+1} > i_a$.  They are non-decreasing since $i_{a+1}-(a+1)\ge i_a+1-(a+1)=i_a-a$.  Finally, since $m-1 \ge i_{g-m+1}$ and $i_{g-m+1} - (g-m+1) \ge i_a-a$, we have $2m-2-g \ge i_a - a$.  Thus each distinct choice of these $i_a$ is associated with a unique partition of $w-(g-m+1)$ into at most $g-m+1$ parts, each of size at most $2m-2-g$. Furthermore, from any such partition $j_1+\dotsb+j_{g-m+1}$, where $0\le j_1\le\dotsb\le j_{g-m+1}$, it is easy to reconstruct $S$ by setting $i_a=j_a+a$; the resulting $i_a$ will be strictly increasing and bounded above by $m-1$, as desired. This completes the proof of the lemma.
\end{proof}

We recall the convention that there is a unique partition of $0$ which has $0$ parts.  This shows that the lemma also holds for semigroups with $F<m$, which all satisfy $F = m-1$.

We observe that $p(x,y,z)=p(yz-x,y,z)$, since if $j_1+\dotsb+j_y=x$ is a partition of $x$ into parts of size at most $z$, then $(z-j_1)+\dotsb+(z-j_y)=yz-x$ is a partition of $yz-x$ into parts of size at most $z$, and vice versa. This simple fact will be useful later.

We now state and prove our main theorem of this section.

\begin{thm}
\label{quad}
Let $\bt_1 = \frac{3}{2} \left(\frac{\ln \varphi}{\pi}\right)^2,\ \gamma = \frac{5+\sqrt{5}}{10},\ \bt_2 =(1-\gamma)(2\gamma-1)-\beta_1$, and $\eps>0$.  Let $Y_{\eps}(g)$ be the number of numerical semigroups with genus $g$ and weight at most $(\bt_1-\eps)g^2$ and $Z_{\eps}(g)$ be the number of numerical semigroups with genus $g$ and weight at least $(\bt_2+\eps)g^2$.  Then $\lim_{g\rar\infty}\frac{Y_{\eps}(g)}{N(g)}=\lim_{g\rar\infty}\frac{Z_{\eps}(g)}{N(g)}=0$.
\end{thm}

In order to show this, we need the Hardy-Ramanujan formula \cite{HR}:
\begin{thm}[Hardy, Ramanujan]
\label{hr}
Let $p(n)$ be the total number of partitions of $n$. Then as $n$ grows large, $p(n)$ is asymptotically equal to
\[
\frac{1}{4n\sqrt3}e^{\pi\sqrt{\frac{2n}3}}.
\]
\end{thm}

\begin{proof}[Proof of Theorem~\ref{quad}.]
We may assume that $F<3m$ by Theorem~\ref{3m}. We consider the cases $F<2m$ and $2m<F<3m$ separately. 

First suppose $F<2m$. Let $K_2(w,m,g)$ be the number of numerical semigroups with genus $g$, weight $w$, multiplicity $m$, and $F<2m$. We wish to bound $Y_{\eps}(g)=\sum_{m=1}^{g}\sum_{w<(\bt_1-\eps)g^2}K_2(w,m,g)$. Now by Lemma~\ref{qbinom}, $K_2(w,m,g)$ is equal to $p(w-(g-m+1),g-m+1,2m-2-g)$, so in particular $K_2(w,m,g)\le p(w-(g-m+1))\le p(w)$ where $p(w)$ is the total number of partitions of $w$. But by Theorem~\ref{hr}, we have
\[
\sum_{m=1}^g\sum_{w<(\bt_1-\eps)g^2}p(w)=O(g e^{\pi\sqrt{2/3}\sqrt{\bt_1-\eps}g})=o(\phi^g)
\]
for $\bt_1$ such that $e^{\pi\sqrt{2/3}\sqrt{\bt_1}}=\phi$, which gives $\bt_1= \frac{3}{2} \left(\frac{\ln\phi}{\pi}\right)^2 \approx0.035$. This gives the desired lower bound on the weight of a typical numerical semigroup.

To show $Z_{\eps}(g)/N(g)$ goes to $0$ as $g$ goes to infinity, we transform the problem of bounding $Z_{\eps}(g)$ into the problem of bounding $Y_{\eps/2}(g)$. Recall that $\gam=\frac{5+\sqrt5}{10}$ and $\bt_2 =(1-\gam)(2\gam-1)-\bt_1$, and define $I := ((\gam-\del)g,(\gam+\del)g)$, where $\del$ is chosen so small that $\del+2\del^2<\eps/2$. From Proposition~\ref{fracmg}, it suffices to consider those semigroups for which $m\in I$. We use Lemma~\ref{qbinom} to write
\[
\sum_{m\in I}\sum_{w>(\bt_2+\eps)g^2}K_2(w,m,g) = \sum_{m\in I}\sum_{w>(\bt_2+\eps)g^2}p(w-(g-m+1),g-m+1,2m-2-g).
\]
As noted above, we have $p(w-(g-m+1),g-m+1,2m-2-g)=p((g-m+1)(2m-2-g)+(g-m+1)-w,g-m+1,2m-2-g)$. Let $w'=(g-m+1)(2m-2-g)+(g-m+1)-w$, so that the right-hand side of the above equation can be rewritten as
\begin{align*}
&\sum_{m\in I}\sum_{w>(\bt_2+\eps)g^2}p(w',g-m+1,2m-2-g)\\
&=\sum_{m\in I}\sum_{w'<(g-m+1)(2m-2-g)+g-m+1-(\bt_2+\eps)g^2}p(w',g-m+1,2m-2-g)\\
&=\sum_{m\in I}\sum_{\substack{w''<(g-m+1)(2m-2-g)\\+2(g-m+1)-(\bt_2+\eps)g^2}}p(w''-(g-m+1),g-m+1,2m-2-g)\\
&=\sum_{m\in I}\sum_{w''<(g-m+1)(2m-2-g)+2(g-m+1)-(\bt_2+\eps)g^2}K_2(w'',m,g)
\end{align*}
where we again used Lemma~\ref{qbinom} in the last line. Assuming from Proposition~\ref{fracmg} that $m\in I$, we have $g-m+1< g-(\gam-\del)g+1=(1-\gam+\del)g+1$ and $2m-2-g<2(\gam+\del)g-g=(2\gam-1+2\del)g$. Hence the last line above is at most
\begin{align*}
&\sum_{m\in I}\sum_{w''<[(1-\gam+\del)g+1](2\gam-1+2\del)g+2(1-\gam+\del)g+2-(\bt_2+\eps)g^2}K_2(w'',m,g)\\
= &\sum_{m\in I}\sum_{w''<[(1-\gam)+\del][(2\gam-1)+2\del]g^2-(\bt_2+\eps)g^2+(1+4\del)g+2}K_2(w'',m,g)\\
= &\sum_{m\in I}\sum_{w''<(\bt_1+\del+2\del^2-\eps)g^2+(1+4\del)g+2}K_2(w'',m,g)\\
\le &\sum_{m\in I}\sum_{w''<(\bt_1-\eps/2)g^2}K_2(w'',m,g)
\end{align*}
for sufficiently large $g$, since $\eps-(\del+2\del^2)>\eps/2$ by construction. But the last line is at most $Y_{\eps/2}(g)=o(\phi^g)$, so we are done.

Next suppose $2m<F<3m$, and suppose the Ap\'ery set is given by $\{k_1 m+1,\ldots, k_{m-1} m + m-1\}$. Let $K_3(w,m,t,g)$ be the number of numerical semigroups with genus $g$, weight $w$, multiplicity $m$, $F$ satisfying $2m<F<3m$, and exactly $t$ values of $a$ such that $k_a=3$. Using Theorem~\ref{pluseps} and Proposition~\ref{fracmg} above, we may assume that $F<(2+\del_0)m$, where $\del_0$ will be chosen later. We first bound the number of semigroups with $w<(\bt_1-\eps)g^2$. The intuition here is that the numerical semigroups with $2m<F<(2+\del_0)m$ and weight $w$ look fairly similar to the numerical semigroups with $F<2m$ and weight $w$, and the ability to set some $k_i$ equal to $3$ does not greatly increase the number of such semigroups.

Let $i_1<\dotsb<i_s$ be the set of indices such that $k_{i_a}\ge2$, and let $i_{j_1}<\dotsb<i_{j_t}$ be the set of indices such that $k_{i_{j_a}}=3$. We have $s+t=g-m+1$, or alternatively, $g=m+s+t-1$. We can write the weight $w$ as follows:
\begin{align*}
w &=\sum_{a=1}^{m-1}a+\sum_{a=1}^s(m+i_a)+\sum_{a=1}^t(2m+i_{j_a})-\sum_{a=1}^{m+s+t-1}a\\
&=\sum_{a=1}^s(m+i_a)-\sum_{a=1}^s(m-1+a)+\sum_{a=1}^t(2m+i_{j_a})-\sum_{a=1}^t(m+s-1+a)\\
&=\sum_{a=1}^s(i_a-a)+s+\sum_{a=1}^t(i_{j_a}-a)+t(m-s+1).
\end{align*}

We set $\sum_{a=1}^s(i_a-a)=d$. We work in parallel to the case $F<2m$. By the same reasoning as in the proof of Lemma~\ref{qbinom}, the number of choices for the $i_a$ is precisely $p(d,s,m-1-s)<p(d)<p(w)$. The total number of choices for the $i_{j_a}$ is at most $\binom st=\binom{g-m+1-t}{t}$. Choose $\eps_0$ small enough that $\psi:=(1+\eps_0)e^{\pi\sqrt{2/3}\sqrt{\bt_1-\eps}}<\phi$. By Proposition~\ref{close1}, there is some $\del_0>0$ so that $\sum_{R=0}^{\lceil \del_0 g \rceil}\binom{g-R}{R}$ is bounded above by a constant times $(1+\eps_0)^g$ for sufficiently large $g$. Since we assumed that $F<(2+\del_0)m$, we have $t\le i_{j_t}\le\del_0 m<\del_0 g$. Then $\binom{g-t}{t}$ is bounded above by $(1+\eps_0)^g$, and so of course $\binom{g-m+1-t}{t}$ is as well. Summing over $d$, we conclude that $K_3(w,m,t,g)$ is bounded by a polynomial in $g$ times $p(w)(1+\eps_0)^g$, which is in turn bounded by $p((\bt_1-\eps)g^2)(1+\eps_0)^g=O\left(\frac{\psi^g}{g^2}\right)$ by the Hardy-Ramanujan formula. Summing over all $m$, $t$, and $w<(\bt_1-\eps)g^2$ gives a count of possible such semigroups $S$ which is $o(\phi^g)$, as desired.

For $w>(\bt_2+\eps)g^2$, we proceed by reducing to the situation $w<(\bt_1-\eps/4)g^2$, again in parallel with our strategy for $F<2m$. Let $0<\del_0<\eps/4$; then for sufficiently large $g$ and all $t<\del_0g$, we have 
\[t(2m-s+1)+s<2mt+t+s\le2gt+t+s<2\del_0g^2+t+s<(\eps/2)g^2.\] Now since $(1-\gam)\del_0<(1-\gam)\eps/4$, we can choose $\del_1$ so small that $\del_1+2\del_1^2+(1-\gam+\del_1)\del_0<\eps/4$. Next, choose $\eps_0$ small enough such that $\psi:=(1+\eps_0)e^{\pi\sqrt{2/3}\sqrt{\bt_1-\eps/4}}<\phi$. Finally, choose $\del_2$ small enough that $\sum_{R=0}^{\lceil \del_2 g \rceil}\binom{g-R}{R}$ is bounded above by $(1+\eps_0)^g$, and let $\del=\min\{\del_0,\del_1,\del_2\}$. Since $i_{j_a} - a < m$, for any $t$ we have $\sum_{a=1}^t (i_{j_a}-a) < tm$.  For $t<\delta g$,  
\[
\sum_{a=1}^s(i_a-a)=w-s-\sum_{a=1}^t(i_{j_a}-a)-t(m-s+1)>w-s-tm-t(m-s+1)>(\bt_2+\eps/2)g^2.
\]
If $d=\sum_{a=1}^s(i_a-a)$, then for a fixed $d$ we have $p(d,s,m-1-s)$ ways of choosing the $i_a$ and no more than $\binom st$ ways of choosing the $i_{j_a}$ from the $i_a$. Therefore we have
\[
K_3(w,m,t,g)\le\sum_{d>(\bt_2+\eps/2)g^2}p(d,s,m-1-s)\binom st
\]
and again using the fact that $p(x,y,z)=p(yz-x,y,z)$, the right-hand side can be rewritten as
\begin{align*}
&\sum_{d<s(m-1-s)-(\bt_2+\eps/2)g^2}p(d,s,m-1-s)\binom st\\
\le\ &(1+\eps_0)^g\sum_{d<s(m-1-s)-(\bt_2+\eps/2)g^2}p(d,s,m-1-s).
\end{align*}
For $t<\del g$, we have $s=g-m+1-t > (1-\del)g-m$. Then further assuming that $m\in((\gam-\del)g,(\gam+\del)g)$, we have $s\le g-m+1<(1-\gam+\del)g+1$ and $m-1-s < 2m-(1-\del)g<(2\gam-1+3\del)g$. Thus the last line can be bounded above by
\begin{align*}
& (1+\eps_0)^g\sum_{d<[(1-\gam+\del)g+1](2\gam-1+3\del)g-(\bt_2+\eps/2)g^2}p(d,s,m-1-s)\\
=\ &(1+\eps_0)^g\sum_{d<(\bt_1-\eps/2)g^2+(\del+2\del^2+(1-\gam+\del)\del)g^2+(2\gam-1+3\del)g}p(d,s,m-1-s).
\end{align*}
Since we chose $\del$ so that $(\del+2\del^2+(1-\gam+\del)\del)g^2<(\eps/4)g^2$, for sufficiently large $g$ this is at most
\begin{align*}
&(1+\eps_0)^g\sum_{d<(\bt_1-\eps/4)g^2}p(d,s,m-1-s)\\
\le\ & (1+\eps_0)^g\sum_{d<(\bt_1-\eps/4)g^2}e^{\pi\sqrt{2/3}\sqrt{\bt_1-\eps/4}\cdot g}=O( \psi^g)
\end{align*}
by the Hardy-Ramanujan formula. Summing over $t$ and $m$ gives the desired bound.
\end{proof}
It would be interesting to see whether we can improve on the constants given in the statement of this result.  A more careful analysis of the partitions occurring in this proof will probably yield better bounds.

Since $g-1<(\beta_1-\eps)g^2$ for, say, $\eps=\beta_1/2$ and sufficiently large $g$, Proposition~\ref{lessg} follows immediately from Theorem~\ref{quad}. This result shows that although the Eisenbud and Harris family of semigroups are Weierstrass, they also have density zero in the entire set of numerical semigroups.  The main theorem of the paper shows that the set of semigroups which are not Weierstrass because they fail the Buchweitz criterion for some $n$ also has density zero.  There are other known examples of semigroups which are known to be Weierstrass but do not fit into this Eisenbud and Harris family, and semigroups which are not Weierstrass but do not fail the Buchweitz criterion.  See for example \cite{Torres, TorresOV}.  

It would be very interesting to show that a positive proportion of numerical semigroups are Weierstrass, or to show that a positive proportion are not Weierstrass.  The problem of determining the density of the set of Weierstrass semigroups in the entire set of numerical semigroups remains completely open.

\section{Acknowledgments}
We would like to thank Nathan Pflueger, Alex Zhai, Gaku Liu, Evan O'Dorney, and Yufei Zhao for helpful conversations. We would also like to thank Jonathan Wang and Alan Deckelbaum for assistance during this project. We thank the referee for many helpful comments that improved the clarity of the paper.  Sections of this work were done as part of the University of Minnesota Duluth REU program, funded by NSF/DMS grant 1062709 and NSA grant H98230-11-1-0224. Finally, we would like to thank Joseph Gallian for his advice and support.

\end{document}